\documentclass[preprint,12pt]{elsarticle}

\usepackage{amsmath,amssymb}
\usepackage{amsthm,color,amsfonts,latexsym,mathptmx}
\usepackage[ansinew]{inputenc}
\usepackage[T1]{fontenc}
\usepackage{tikz-cd}
\usepackage[all]{xy}

\newtheorem{theorem}{Theorem}[section] 
\newtheorem{lemma}[theorem]{Lemma}     
\newtheorem{corollary}[theorem]{Corollary}
\newtheorem{proposition}[theorem]{Proposition}
\newtheorem{definition}[theorem]{Definition}
\newtheorem{example}[theorem]{Example}

\journal{Journal of Mathematical Analysis and Applications}

\begin{document}

\begin{frontmatter}


\title{Positive Operator Valued Measures and Feller Markov Kernels
}
 \author{R. Beneduci
\fnref{label2}}



 \address[label2]{roberto.beneduci@unical.it}


\address{Dipartimento di Fisica, Universit\'a della Calabria\\ and\\ Istituto Nazionale di Fisica Nucleare gruppo c. Cosenza,\\ via P. Bucci cubo 30-B, 87036 Arcavacata di Rende (Cosenza) Italy}

\begin{abstract}
\noindent

A Positive Operator Valued Measure (POVM) is a map $F:\mathcal{B}(X)\to\mathcal{L}_s^+(\mathcal{H})$ from the Borel $\sigma$-algebra of a topological space $X$ to the space of positive self-adjoint operators on a Hilbert space $\mathcal{H}$. We assume $X$ to be  Hausdorff, locally compact and second countable and prove that a POVM $F$ is commutative if and only if it is the smearing of a spectral measure $E$ by means of a Feller Markov kernel. Moreover, we prove that the smearing can be realized by means of a strong Feller Markov kernel if and only if $F$ is uniformly continuous. Finally, we prove that a POVM which is norm bounded by a finite measure $\nu$ admits a strong Feller Markov kernel. 

That provides a characterization of the smearing which connects a commutative POVM $F$ to a spectral measure $E$ and is relevant both from the mathematical and the physical viewpoint since smearings of spectral measures form a large and very relevant subclass of POVMs: they are paradigmatic for the modeling of certain standard forms of noise in quantum measurements, they provide optimal approximators as marginals in joint measurements of incompatible observables \cite{Busch}, they are important for a range of quantum information processing protocols, where classical post-processing plays a role \cite{Heinosaari}.

The mathematical and physical relevance of the results is discussed and particular emphasis is given to the connections between the Markov kernel and the imprecision of the measurement process.


\end{abstract}

\begin{keyword}
Positive Operator Valued Measures\sep Feller Markov kernels\sep $C^*$-algebras\sep von Neumann algebras\sep Quantum observables\sep Quantum Measurement



\end{keyword}

\end{frontmatter}


\section{Introduction}
A Positive operator Valued measure (or POVM) is a map $F:\mathcal{B}(X)\to\mathcal{L}_s^+(\mathcal{H})$ from the Borel $\sigma$-algebra of a topological space $X$ to the space of positive self-adjoint operators on a Hilbert space $\mathcal{H}$. In the present paper we assume $X$ to be Hausdorff, locally compact and second countable. If $F(\Delta)$ is a projection operator for each $\Delta\in\mathcal{B}(X)$, $F$ is called Projection Valued measure (or PVM). If $X=\mathbb{R}$ we have real POVMs (or semispectral measures) and real PVMs (or spectral measures) respectively. Therefore, the set of PVMs is a subset of the set of POVMs and the set of spectral measures is a subset of the set of semispectral measures. Moreover, spectral measures are in one-to-one correspondence with self-adjoint operators (spectral theorem) and are used in standard quantum mechanics to represent quantum observables. It was pointed out \cite{Ali,Bush,Davies,Holevo1,Prugovecki,Schroeck} that POVMs are more suitable than spectral measures in representing quantum observables. The quantum observables described by POVMs are called generalized observables or unsharp observables and play a key role in quantum information theory, quantum optics, quantum estimation theory \cite{Bush,Helstrom,Holevo1,Schroeck1} and in the phase space formulation of quantum mechanics \cite{Prugovecki,Schroeck1,Guz,Davies,B9,B10}. It is then natural to ask what are the relationships between POVMs and spectral measures. A clear answer can be given in the commutative case \cite{Ali,B0,B1,B2,B3,B4,B5,B6,B7,B8,Holevo,P}. Indeed \cite{B1,P}, a POVM $F$ is commutative if and only if there exist a bounded self-adjoint operator $A$ and a Markov kernel (transition probability) $\mu_{(\cdot)}(\cdot):\sigma(A)\times\mathcal{B}(X)\to[0,1]$ such that 
$$F(\Delta)=\int_{\sigma(A)}\mu_{\Delta}(\lambda)\,dE_{\lambda}$$
\noindent
 where, $E$ is the spectral measure corresponding to $A$. 
In other words, $F$ is a smearing of the spectral measure $E$ corresponding to $A$.

 Smearings of spectral measures form a large and very relevant subclass of POVMs and are paradigmatic for the modeling of certain standard forms of noise in measurements. They also provide optimal approximators as marginals in joint measurements of incompatible observables (for example, for position and momentum) as shown by Busch, Lahti, Werner in Ref. \cite{Busch}. Moreover, they are important for a range of quantum information processing protocols, where classical post-processing plays a role \cite{Heinosaari}.  
 Another relevant application of commutative POVMs is the smearing of incompatible observables in order to get compatible observables (see \cite{Busch1,B11}). All that explains the relevance of commutative POVMs both form the mathematical and the physical viewpoint. As a notable example we analyze (section \ref{PositionMomentum}) the unsharp position and momentum observables which are the marginals of a joint position momentum observable (see \cite{Schroeck,Bush,Busch}). 

Although, it is well known that $F$ can be interpreted as the smearing of $E$, no characterization of the smearing (the Markov kernel) is known. In the present paper such a characterization is given and its mathematical and physical implications are analyzed. That also provides a stronger characterization of commutative POVMs by means of Feller Markov kernels.

In order to outline some of the problems we deal with, it is helpful to consider the unsharp position observable in the interval $[0,1]$. It can be represented as follows. 
\begin{align}\label{PositionExample}
\langle\psi,Q^f(\Delta)\psi\rangle&:=\int_{[0,1]}\mu_{\Delta}(\lambda)\,d\langle\psi,Q_\lambda\psi\rangle,\quad\Delta\in\mathcal{B}(\mathbb{R}),\quad\psi\in L^2([0,1]),\\
\mu_{\Delta}(\lambda)&:=\int_{\mathbb{R}}\chi_{\Delta}(\lambda-y)\, f(y)\,dy,\quad \lambda\in[0,1]\notag
\end{align}
\noindent
where, $f$ is a positive, bounded, Borel function such that $f(y)=0$, $y\notin [0,1]$, and $\int_{[0,1]} f(y)dy=1$, while $Q_\lambda$ is the spectral measure corresponding to the position operator
\begin{align*}
Q:L^2([0,1])&\to L^2([0,1])\\
(Q\psi)(x)&:=x\psi(x)
\end{align*}
\noindent
for almost all $x\in[0,1]$. We recall that $\langle\psi,Q(\Delta)\psi\rangle$ is interpreted as the probability that a perfectly accurate measurement (sharp measurement) of the position gives a result in $\Delta$.
\noindent
Then, a possible interpretation of equation (\ref{PositionExample}) is that $Q^f$ is a randomization of $Q$. Indeed \cite{Prugovecki}, the outcomes of the measurement of the position of a particle depend on the measurement imprecision\footnote{There are other possible interpretations of the randomization. For example, it could be due to the existence of a no-detection probability depending on hidden variables \cite{Garola}.} so that, if the sharp value of the outcome of the measurement of $Q$ is $\lambda$ then the apparatus produces with probability $\mu_{\Delta}(\lambda)$ a reading in $\Delta$. 

It is worth noting that (see example \ref{example2} in section 5) the Markov kernel 
$$\mu_{\Delta}(\lambda):=\int_{\mathbb{R}}\chi_{\Delta}(\lambda-y)\, f(y)\,dy,\quad \lambda\in[0,1]$$
\noindent
 in equation (\ref{PositionExample}) above is such that the function $\lambda\mapsto\mu_{\Delta}(\lambda)$ is continuous for each $\Delta\in\mathcal{B}(\mathbb{R})$. The continuity of $\mu_{\Delta}$ means that if two sharp values $\lambda$ and $\lambda'$ are very close to each other then, the corresponding random diffusions are very similar, i.e., the probability to get a result in $\Delta$ if the sharp value is $\lambda$ is very close to the probability to get a result in $\Delta$ if the sharp value is $\lambda'$. That is quite common in important  physical applications and  seems to be reasonable from the physical viewpoint. It is then natural to look for general conditions which ensure the continuity of $\lambda\mapsto\mu_{\Delta}(\lambda)$. That is one of the aims of the present work. What we prove is that, in general, the continuity does not hold for all the Borel sets $\Delta$ but only for a ring of subsets which generates the Borel $\sigma$-algebra $\mathcal{B}(X)$. (Anyway, that is sufficient to prove the weak convergence of $\mu_{(\cdot)}(\lambda)$ to $\mu_{(\cdot)}(\lambda')$.) We also prove that the continuity for each Borel set is equivalent to the uniform continuity of $F$ which in its turn is equivalent to require that the smearing in equation (\ref{PositionExample}) can be realized by a strong Feller Markov kernel.

 It is our opinion that, in the real case, the continuity of $\mu_{\Delta}$ over a ring $\mathcal{R}$ which generates the Borel $\sigma$-algebra of the reals could be helpful in dealing with problems connected to the characterization of functions of the kind 

$$G_f(\lambda)=\int_{\mathbb{R}} f(t)\,\mu_{dt}(\lambda).$$ 

\noindent
A similar (but less general) problem  arises in Ref. \cite{B6} where the relationships between  Naimark extension theorem and the characterization of commutative POVMs as smearing of spectral measures are analyzed.
That is a second motivation for the analysis of the continuity properties of $\mu_{\Delta}$. 

The results outlined above are a consequence of the two main theorems of the present work. 

\noindent
The first is a characterization of the smearing which connects a commutative POVM to a real PVM. In particular, we show (see theorems \ref{weak}) that a POVM is commutative if and only if  there exist a spectral measure $E$ and a  Feller Markov kernel $\mu_{(\cdot)}(\cdot):\Gamma\times\mathcal{B}(X)\to[0,1]$, $\Gamma\subset\sigma(A)$, $E(\Gamma)=\mathbf{1}$, such that 

\begin{equation}\label{Feller}
 F(\Delta)=\int_{\Gamma}\mu_{\Delta}(\lambda)\,dE_{\lambda}
\end{equation}

\noindent
and $\mu_{\Delta}(\cdot)$ is continuous for each $\Delta\in\mathcal{R}$ where, $\mathcal{R}\subset\mathcal{B}(X)$ is a ring which generates the Borel $\sigma$-algebra $\mathcal{B}(X)$ and $A$ is the self-adjoint operator corresponding to $E$. 
Therefore, $F$ is commutative if and only if there exists a Feller Markov kernel $\mu$ such that equation (\ref{Feller}) is satisfied. See section \ref{4} for the definition of Feller Markov kernel. That provides a new and stronger characterization of commutative POVMs. 

\noindent 
We also prove that the family of functions $\{\mu_{\Delta}\}_{\Delta\in\mathcal{B}(X)}$ separates the points of $\sigma(A)$ up to a null set (see theorems \ref{separate}, and \ref{weak}). In other words, the probability measures $\mu_{(\cdot)}(\lambda)$ and $\mu_{(\cdot)}(\lambda')$ which represent the randomizations corresponding to the sharp values $\lambda$ and $\lambda'$, $\lambda\neq\lambda'$, are different. 


The second theorem is a characterization of the POVMs which admit a strong Feller Markov kernel, i.e., a Markov kernel $\mu$ such that the function $\lambda\mapsto\mu_{\Delta}(\lambda)$ is continuous for each $\Delta\in\mathcal{B}(X)$. In particular, we prove (see theorem \ref{uni}) that a POVM $F$ admits a strong Feller Markov kernel if and only if it is uniformly continuous. 
As an example, we develop the details for the unsharp position observable defined in equation (\ref{PositionExample}) above. Finally, we prove (see section 6) that a POVM $F$ which is norm bounded by a regular finite measure $\nu$ is uniformly continuous (theorem \ref{abs}). We give some examples of POVMs that are norm bounded by regular measures (see example \ref{PositionExample3}) and analyze the unsharp position observable which is obtained as the marginal of a phase space observable (see section 6.1).

\section{Some preliminaries about POVMs}
\noindent
In what follows, we denote by $\mathcal{B}(X)$ the Borel $\sigma$-algebra of a topological space $X$, by $\textbf{0}$ and $\textbf{1}$ the null and the identity operators, by $\mathcal{L}_s(\mathcal{H})$ the space of all bounded self-adjoint linear operators acting in a Hilbert space $\mathcal{H}$ with scalar product $\langle\cdot,\cdot\rangle$, by $\mathcal{L}_s^+(\mathcal{H})$ the subspace of all positive, bounded self-adjoint operators on $\mathcal{H}$, by $\mathcal{E}(\mathcal{H})\subset\mathcal{L}_s^+(\mathcal{H})$ the subspace of all projection operators on $\mathcal{H}$. We use the symbol $\mathcal{C}(\Lambda)$ to denote the algebra of continuous functions on $\Lambda$.

\begin{definition}
\label{POV}
A Positive Operator Valued measure (for short, POVM) is a map $F:\mathcal{B}(X)\to\mathcal{L}_s^+(\mathcal{H})$
such that:
    \begin{equation*}
    F\big(\bigcup_{n=1}^{\infty}\Delta_n\big)=\sum_{n=1}^{\infty}F(\Delta_n).
    \end{equation*}
    \noindent 
 where, $\{\Delta_n\}$ is a countable family of disjoint
    sets in $\mathcal{B}(X)$ and the series converges in the weak operator topology. It is said to be normalized if 
\begin{equation*}   
    F(X)={\bf{1}}
\end{equation*}
\end{definition}    
\begin{definition}
    A POVM is said to be commutative if
    \begin{equation}
    \big[F(\Delta_1),F(\Delta_2)\big]={\bf{0}},\,\,\,\,\forall\,\Delta_1\,,\Delta_2\in\mathcal{B}(X).
    \end{equation}
    \end{definition}

   \begin{definition}
   A POVM is said to be orthogonal if
    \begin{equation}
    F(\Delta_1)F(\Delta_2)={\bf{0}}\,\,\,\hbox{if}\,\,\Delta_1\cap\Delta_2=
    \emptyset.
    \end{equation}
\end{definition}
\begin{definition}
A Projection Valued measure (for short, PVM) is an orthogonal, normalized POVM.
\end{definition}
\noindent
It is simple to see that for a PVM $E$, we have $E(\Delta)=E(\Delta)^2$, for any $\Delta \in \mathcal{B}(X)$. Then, $E(\Delta)$ is a projection operator for every $\Delta\in\mathcal{B}(X)$, and the PVM is a map $E:\mathcal{B}(X)\to\mathcal{E}(\mathcal{H})$.

In quantum mechanics, non-orthogonal normalized POVMs are also called \textbf{generalised} or \textbf{unsharp} observables and PVMs \textbf{standard} or \textbf{sharp} observables. 
\noindent
\begin{definition}\label{spectrum}
The spectrum $\sigma(F)$ of a \textit{POVM} $F$ is the set of points $x\in X$ such that $F(\Delta)\neq\mathbf{0}$, for any open set $\Delta$ containing $x$.
\end{definition}
\noindent
The spectrum $\sigma(F)$ of a POVM $F$ is a closed set since its complement $X-\sigma(F)$ is the union of all the open sets $\Delta\subset X$ such that $F(\Delta)=\mathbf{0}$. 

A spectral measure is a real PVM, i.e., a PVM $E$ such that $\sigma(E)\subset\mathbb{R}$.

\noindent

\begin{definition}
The von Neumann algebra $\mathcal{A}^W(F)$ generated by the POVM $F$ is the von Neumann algebra generated by the set $\{F(\Delta)\}_{\Delta\in\mathcal{B}(X)}$. 
\end{definition}

In the following we use the symbols $w-\lim$ and $u-\lim$ to denote the limit in the weak operator topology and the limit in the uniform operator topology respectively.

\begin{definition}
 A POVM is regular if for every $\Delta\in\mathcal{B}(X)$,
\begin{align*}
F(\Delta)&=GLB\big\{F(G)\,:\,\Delta\subset G,\,\, G\in\mathcal{B}(X),\,\, G\,\,\,\,\text{open}\big\}\\
F(\Delta)&=LUB\big\{F(C)\,:\,C\subset\Delta,\,\, C\in\mathcal{B}(X),\,\, C\,\,\,\,\text{compact}\big\}
\end{align*}

\end{definition}
\noindent
\begin{proposition}
A POVM defined on a Hausdorff locally compact, second countable space $X$ is regular. 
\end{proposition}
\begin{proof}

Since $X$ is metrizable and $\sigma$-compact, the ring of Borel sets coincides with the ring of Baire sets 
and the thesis comes from the fact that each Baire POVM is regular \cite{Berberian}.

\end{proof}

In what follows, we use the term ``measurable'' for the Borel measurable functions.
For any vector $\psi\in\mathcal{H}$ the map
$$\langle F(\cdot)\psi,\psi\rangle \,:\,\mathcal{B}(X)\to {\mathbb R} ,
\qquad
\Delta \mapsto \langle F(\Delta)\psi,\psi\rangle,$$
is a measure. 
 In the following we will use the symbol $d\langle F_{\lambda}\psi,\psi\rangle$ to mean integration with respect to the measure $\langle F(\cdot)\psi,\psi\rangle$.
\noindent
We shall say that a measurable function $f:X\to\mathbb{R}$, is almost everywhere (a.e.) one-to-one with respect to a POVM $F$ if it is one-to-one on a subset $N\subset X$ such that $X-N$ is a null set with respect to $F$.
We shall say that a function $f: X\to{\mathbb R}$ is bounded with respect to a POVM $F$, if it is equal to a bounded function $g$ a.e. with respect to $F$, that is, if $f=g$ a.e. with respect to the measure $\langle F(\cdot)\psi,\psi\rangle$,  $\forall \psi \in \mathcal{H}$.
\noindent
For any real, bounded and measurable function $f$ and for any POVM $F$, there is a unique \cite{Berberian} bounded self-adjoint operator $B\in\mathcal{L}_s(\mathcal{H})$ such that
\begin{equation}
\label{6}
\langle B\psi,\psi\rangle=\int_X f(\lambda)d\langle F_{\lambda}\psi,\psi\rangle,\quad\text{for each}\quad \psi\in\mathcal{H}.
\end{equation}
If equation (\ref{6}) is satisfied, we write $B=\int f(\lambda)dF_{\lambda}$ or $B=\int f(\lambda)F(d\lambda)$ equivalently. In the case of a function $f$ which is not bounded with respect to $F$, integration can still be defined but in general it gives a symmetric operator and not a self-adjoint operator (see Ref. \cite{Lahti} for the details). 

By the spectral theorem, there is a one-to-one correspondence between real PVMs  $E$ and self-adjoint operators $B$,
the correspondence being given by
$$B=\int_{\mathbb{R}} \lambda dE^B_{\lambda}.$$
\noindent
Notice that the spectrum of $\sigma(E^B)$ of $E^B$ coincides with the spectrum $\sigma(B)$ of $B$. 
Moreover, in this case a functional calculus can be developed. Indeed, if $f:{\mathbb R}\to{\mathbb R}$ is a measurable real-valued function, we can define the self-adjoint operator 
\begin{equation}\label{functionalcalculus}
f(B)=\int_{\mathbb{R}} f(\lambda) dE^B_{\lambda}
\end{equation}
\noindent
where, $E^B$ is the PVM corresponding to $B$. If $f$ is bounded, then $f(B)$ is bounded. 
Equation (\ref{functionalcalculus}) cannot be extended to the case of non-orthogonal POVMs. 
\noindent

\noindent
In the following we do not distinguish between real PVMs and the corresponding self-adjoint operators.


\noindent
Let $\Lambda$ be a subset of $\mathbb{R}$ and $\mathcal{B}(\Lambda)$ the corresponding Borel $\sigma$-algebra.
\begin{definition}
A real Markov kernel is a map $\mu:\Lambda\times\mathcal{B}(X)\to[0,1]$ such that,
\begin{itemize}
\item[1.] $\mu_{\Delta}(\cdot)$ is a measurable function for each $\Delta\in\mathcal{B}(X)$,
\item[2.] $\mu_{(\cdot)}(\lambda)$ is a probability measure for each $\lambda\in \Lambda$.
\end{itemize}
\end{definition}
\begin{definition}
Let $\nu$ be a measure on $\Lambda$. A map $\mu:\Lambda\times\mathcal{B}(X)\to[0,1]$ is a weak Markov kernel with respect to $\nu$ if:
\begin{itemize}
\item[1.] $\mu_{\Delta}(\cdot)$ is a measurable function for each $\Delta\in\mathcal{B}(X)$,
\item[2.] for every $\Delta\in\mathcal{B}(X)$,\,\,$0\leq\mu_{\Delta}(\lambda)\leq 1$,\quad $\nu-a.e.$,
\item[3.]$\mu_{X}(\lambda)=1$, $\mu_{\emptyset}(\lambda)=0$,\quad $\nu-a.e.$,
\item[4.] for any sequence $\{\Delta_i\}_{i\in\mathbb{N}}$, $\Delta_i\cap\Delta_j=\emptyset$,  
$$\sum_i\mu_{(\Delta_i)}(\lambda)=\mu_{(\cup_i\Delta_i)}(\lambda),\quad \nu-a.e.$$
\end{itemize}
\end{definition}
\begin{definition}
The map $\mu:\Lambda\times\mathcal{B}(X)\to[0,1]$ is a weak Markov kernel with respect to a PVM $E:\mathcal{B}(\Lambda)\to\mathcal{E(H)}$ if it is a weak Markov kernel with respect to each measure $\nu^{\psi}(\cdot):=\langle E(\cdot)\,\psi,\psi\rangle$, $\psi\in\mathcal{H}$. 
\end{definition}
\noindent
In the following, by a weak Markov kernel $\mu$ we mean a weak Markov kernel with respect to a PVM $E$. Moreover the function $\lambda\mapsto\mu_{\Delta}(\lambda)$ will be denoted indifferently by $\mu_{\Delta}$ or $\mu_{\Delta}(\cdot)$.    
\begin{definition}
A POVM $F:\mathcal{B}(X)\to\mathcal{L}_s^+(\mathcal{H})$ is said to be a smearing of a POVM $E:\mathcal{B}(\Lambda)\to\mathcal{L}_s^+(\mathcal{H})$ if there exists a weak Markov kernel $\mu:\Lambda\times\mathcal{B}(X)\to[0,1]$ such that,
\begin{equation*}
F(\Delta)=\int_{\Lambda} \mu_{\Delta}(\lambda)d E_{\lambda}, \,\,\,\,\,\,\,\Delta\in\mathcal{B}(X).
\end{equation*}
\end{definition}

\begin{example}
In the standard formulation of quantum mechanics, the operator
\begin{align*}
Q:\mathcal{D}(Q)&\to L^2(\mathbb{R}),\\
(Q\psi)(x)&:=x\psi(x),
\end{align*}
\noindent
for almost all $x\in\mathbb{R}$, with 
$\mathcal{D}(Q)=\{\psi\in L^2(\mathbb{R})\,\vert\,\int_{\mathbb{R}}\, x^2|\psi(x)|^2\,dx<\infty\}$,
\noindent
is used to represent the position observable. A more realistic description of the position observable of a quantum particle is given by a smearing of $Q$ as, for example, the optimal position POVM
\begin{align*}
F^Q(\Delta)&=\frac{1}{l\,\sqrt{2\,\pi}}\int_{-\infty}^{\infty}\Big(\int_{\Delta}e^{-\frac{(\lambda-y)^2}{2\,l^2}}\,d y\Big)\,dE^Q_\lambda=\int_{-\infty}^{\infty}\mu_{\Delta}(\lambda)\,dE^Q_\lambda
\end{align*}
\noindent
where, 
\begin{equation*}
\mu_{\Delta}(\lambda)=\frac{1}{l\,\sqrt{2\,\pi}}\int_{\Delta}e^{-\frac{(\lambda-y)^2}{2\,l^2}}\,dy 
\end{equation*}
\noindent
defines a Markov kernel and $E^Q$ is the spectral measure corresponding to the position operator $Q$.
\end{example}
\noindent
In the following, the symbol $\mu$ is used to denote both  Markov kernels and weak Markov kernels. The symbols $A$ and $B$ are used to denote self-adjoint operators. 
\begin{definition}
Whenever $F$, $A$, and $\mu$ are such that $F(\Delta)=\mu_{\Delta}(A)$, $\Delta\in\mathcal{B}(X)$, we say that $(F,A,\mu)$ is a von Neumann triplet. 
\end{definition}
\noindent
The following theorem establishes a relationship between commutative POVMs and spectral measures and gives a characterization of the former. Other characterizations and an analysis of the relationships between them can be found in Ref.s \cite{Ali,Holevo,Ali5,P1}.
\begin{theorem}[\cite{B1,P}]\label{Cha}
 A POVM $F$ is commutative if and only if there exist a bounded self-adjoint operator $A$ and a Markov kernel (weak Markov kernel) $\mu$ such that $(F,A,\mu)$ is a von Neumann triplet. 
\end{theorem}
\begin{corollary}\label{smearing}
 A POVM $F$ is commutative if and only if it is a smearing of a real PVM $E$ with bounded spectrum.
\end{corollary}
\begin{definition}
If $A$ and $F$  in theorem \ref{Cha} generate the same von Neumann algebra then $A$ is named the sharp version of $F$.
\end{definition}
\begin{theorem}\label{unique}\cite{B1}
The sharp version $A$ is unique up to almost everywhere bijections. 
\end{theorem}

\section{On the separation properties of $\mu$}
\noindent

In the following, we assume $X$ to be Hausdorff, locally compact and second countable. The symbol $\mathcal{S}$ denotes a countable basis for the topology of $X$. The symbol $\mathcal{R(S)}$ denotes the ring generated by $\mathcal{S}$. Notice that $\mathcal{R(S)}$ is countable and generates the Borel $\sigma$-algebra $\mathcal{B}(X)$.

A weak Markov kernel $\mu$ such that $(F,A,\mu)$ is a von Neumann triplet, separates the points of $\Gamma\subset\sigma(A)$ if the family of functions $\{\mu_{\Delta}\}_{\Delta\in\mathcal{B}(X)}$ separates the points of $\Gamma$ or, in other words, if $\lambda\neq\lambda'$ implies $\mu_{(\cdot)}(\lambda)\neq\mu_{(\cdot)}(\lambda')$. It is then natural to ask if in general $\mu$ has that property. The following theorem answers in the positive.
\begin{theorem}\label{separate}
Let $(F,A,\mu)$ be a von Neumann triplet and suppose that $A$ is a sharp version of $F$. Then, there exists a set $\Gamma\subseteq\sigma(A)$, $E^A(\Gamma)=\mathbf{1}$, such that the family of functions $\{\mu_{\Delta}(\cdot)\}_{\Delta\in\mathcal{B}(X)}$ separates the points of $\Gamma$.
\end{theorem} 
\begin{proof}
In the following, $O_2:=\{F(\Delta)\}_{\Delta\in\mathcal{R(S)}}$ and $\mathcal{A}^C(O_2)$ is the $C^*$-algebra generated by $O_2$. The von Neumann algebra generated by $\mathcal{A}^C(O_2)$ coincides with $\mathcal{A}^W(F)$ (see appendix A). Moreover, $\mathcal{A}^W(F)=\mathcal{A}^W(A)$ since $A$ is the sharp version of $F$.
By the Gelfand-Naimark theorem and the spectral theorem for representations of commutative $C^*$-algebras, there is an $*$-isomorphism $\phi$, 
\begin{equation*}
\mathcal{C}(\Lambda_2)\ni f\mapsto\phi(f)=\int_{\Lambda_2} f(\lambda)\,d\widetilde{E}_{\lambda}
\end{equation*}
\noindent
between $\mathcal{C}(\Lambda_2)$ and $\mathcal{A}^C(O_2)$, where, $\Lambda_2$ is the spectrum of $\mathcal{A}^C(O_2)$ and $\widetilde{E}$ is the spectral measure from $\mathcal{B}(\Lambda_2)$ to $\mathcal{E(H)}$ corresponding to $\phi$. 
 The Gelfand-Naimark isomorphism $\phi$ can be extended to a homomorphism between the algebra of the Borel functions on $\Lambda_2$ and the von Neumann algebra $\mathcal{A}^W(A)$. 
Therefore, there is a Borel function $h$ such that 
\begin{equation}
A=\int_{\Lambda_2} h(\lambda)\,d\widetilde{E}_{\lambda}
\end{equation} 
\noindent
Let $\mathcal{S}$ be a countable basis for the topology of $X$. Let $\{\Delta_i\}_{i\in\mathbb{N}}$ denote an enumeration of the set $\mathcal{R(S)}$. Since $\mathcal{A}^C(O_2)$ and $\mathcal{C}(\Lambda_2)$ are *-isomorphic, the set $\{\nu_{\Delta_i}:=\phi^{-1}(F(\Delta_i))\}_{i\in\mathbb{N}}$ generates $\mathcal{C}(\Lambda_2)$ 
and, by the Stone-Weierstrass theorem, it separates the points in $\Lambda_2$. 

\noindent
Moreover, since $(F,A,\mu)$ is a von Neumann triplet, for each $\Delta_i\in\mathcal{R(S)}$, there is a Borel function $\mu_{\Delta_i}$ such that 
\begin{equation*}
\int_{\Lambda_2}\nu_{\Delta_i}(\lambda)\,d\widetilde{E}_{\lambda}=F(\Delta_i)=\mu_{\Delta_i}(A)=\int_{\Lambda_2}\mu_{\Delta_i}(h(\lambda))\,d\widetilde{E}_{\lambda}.
\end{equation*}
\noindent
Then, for each $\Delta_i\in\mathcal{R(S)}$, there is a set $M_i\subset\Lambda_2$, $\widetilde{E}(M_i)=\mathbf{1}$, such that
\begin{equation}
\mu_{\Delta_i}(h(\lambda))=\nu_{\Delta_i}(\lambda),\quad \lambda\in M_i.
\end{equation}
\noindent
Let $M:=\cap_{i=1}^{\infty} M_i$. Then, 
\begin{equation*}
\widetilde{E}(M)=\lim_{n\to\infty}\widetilde{E}(\cap_{i=1}^{n} M_i)=\lim_{n\to\infty}\prod_{i=1}^{n}\widetilde{E}(M_i)=\mathbf{1}
\end{equation*}
\noindent
and, for each $i\in\mathbb{N}$,
\begin{align}\label{2}
(\mu_{\Delta_i}\circ h)(\lambda)=\nu_{\Delta_i}(\lambda),\quad\lambda\in M\subseteq \Lambda_2. 
\end{align}
\noindent
Since $\{\nu_{\Delta_i}\}_{i\in\mathbb{N}}$ separates the points in $\Lambda_2$, 
$\{\mu_{\Delta_i}\}_{i\in\mathbb{N}}$ separates the points in $\Gamma:=h(M)$. 
\noindent
Moreover\footnote{ Notice that $h(M)$ is a Borel set. In order to prove that, we first recall that $\Lambda_2$ is a Polish space (that is, a complete, separable, space). Indeed, by theorem 11, page 871, in Ref. \cite{Dunford}, it is homeomorphic to a closed subspace of the Cartesian product $\prod_{i=1}^{\infty}\sigma(F(\Delta_i))$, where $\sigma(F(\Delta_i))$ is a complete separable metric space, and by theorem 2, page 406, and theorem 6, page 156, in Ref. \cite{Kuratowski1}, it is complete and separable. Moreover, $h$ is measurable and injective on $M$. Therefore, Souslin's theorem assures that $h(M)$ is a Borel set.},

\begin{align*}
E^A(\Gamma)=E^A(h(M))=\widetilde{E}[h^{-1}(h(M))]=\mathbf{1}
\end{align*}
\noindent
where, $E^A$ is the spectral measure defined by the relation
$$E^A(\Delta)=\widetilde{E}(h^{-1}(\Delta))$$
\noindent
and such that,
$$A=\int \lambda\,dE^A_{\lambda}$$
\noindent
while, $h^{-1}(h(M))$ is a Borel set containing $M$.

\noindent
\end{proof}

\section{Characterization of Commutative POVMs by means of Feller Markov kernels}\label{4}
\noindent
In the present section we introduce the concept of strong Markov kernel, i.e., a weak Markov kernel $\mu_{(\cdot)}(\cdot):\Lambda\times\mathcal{B}(X)\to[0,1]$ with respect to a PVM $E:\mathcal{B}(\Lambda)\to\mathcal{E(H)}$ such that $\mu_{(\cdot)}(\lambda)$ is a probability measure for each $\lambda\in \Gamma\subset \Lambda$, $E(\Gamma)=\mathbf{1}$. Then, we prove (theorem \ref{weak}) that 
 $F$ is commutative if and only if there exists  a bounded self-adjoint operator $A$ and a Feller Markov kernel $\mu$ such that 
$$F(\Delta)=\int_{\Gamma}\mu_{\Delta}(\lambda)\,dE_{\lambda}.$$
 Moreover, we prove that there is a ring $\mathcal{R}$ which generates $\mathcal{B}(X)$ such that $\mu_{\Delta}$ is continuous for each $\Delta\in \mathcal{R}$, and the family of functions $\{\mu_{\Delta}\}_{\Delta\in \mathcal{R}}$ separates the points in $\Gamma$  (see theorems \ref{separate} and \ref{weak}).

 In order to prove the main theorem we need the following definitions.
\begin{definition}
Let $E:\mathcal{B}(\Lambda)\to\mathcal{E(H)}$ be a PVM. The map $\mu_{(\cdot)}(\cdot):\Lambda\times\mathcal{B}(X)\to[0,1]$ is a strong Markov kernel with respect to $E$ if it is a weak Markov kernel with respect to $E$ and there exists a set $\Gamma\subset\Lambda$, $E(\Gamma)=\mathbf{1}$, such that $\mu_{(\cdot)}(\cdot):\Gamma\times\mathcal{B}(X)\to[0,1]$ is a Markov kernel. A strong Markov kernel is denoted by the symbol $(\mu, E,\Gamma\subset\Lambda)$.
\end{definition} 
 \begin{definition}
A Feller Markov kernel is a  Markov kernel $\mu_{(\cdot)}(\cdot):\Lambda\times\mathcal{B}(X)\to[0,1]$ such that the function 
$$G(\lambda)=\int_{X}f(x)\,\mu_{dx}(\lambda),\quad\lambda\in\Lambda$$
\noindent
is continuous and bounded whenever $f$ is continuous and bounded. 
\end{definition}
\begin{theorem}\label{weak}
A POVM $F:\mathcal{B}(X)\to\mathcal{F(H)}$ is commutative if and only if, there exists a bounded self-adjoint operator $A=\int \lambda\,dE_\lambda$ with spectrum $\sigma(A)\subset[0,1]$ and a strong Markov Kernel  $(\mu, E, \Gamma\subset\sigma(A))$ 
such that:
\begin{enumerate}
\item[1)] $\mu_{\Delta}(\cdot):\sigma(A)\to[0,1]$ is continuous for each $\Delta\in\mathcal{R(S)}$,
\item[2)] $F(\Delta)=\int_{\Gamma}\mu_{\Delta}(\lambda)\,dE_{\lambda},\quad\Delta\in\mathcal{B}(X)$.
\item[3)] $\mathcal{A}^W(F)=\mathcal{A}^W(A)$.
\item[4)] $\mu$ separates the points in $\Gamma$.
\end{enumerate}
\noindent
Moreover, $\mu:\Gamma\times\mathcal{B}(X)\to[0,1]$ is a Feller Markov kernel.

\end{theorem}
\begin{proof}
$\mathcal{A}^W(F)$ coincides with the von Neumann algebra generated by the set $O_2:=\{F(\Delta)\}_{\Delta\in\mathcal{R(S)}}$ where, $\mathcal{R(S)}\subset\mathcal{B}(X)$ is the ring generated by $\mathcal{S}$, the countable basis for the topology of $X$ (see appendix A for the proof). We recall that both $\mathcal{S}$ and $\mathcal{R(S)}$  are countable. 

Now, we proceed to the proof of the existence of $A$. Let $\{\Delta_i\}_{i\in\mathbb{N}}$ be an enumeration of the set $\mathcal{R(S)}$. Let $E^{(i)}$ denote the spectral measure corresponding to $F(\Delta_i)\in O_2$. We have $F(\Delta_i)=\int x\,d E^{(i)}_x$. Therefore, for each $i,k\in\mathbb{N}$ there exists a division $\{\Delta_j^{(i,k)}\}_{j=1,\dots,m_{i,k}}$ of $[0,1]$ such that 
\begin{equation}\label{D}
\big\|\sum_{j=1}^{m_{i,k}}x^{(i,k)}_j\,E^{(i)}(\Delta_j^{(i,k)})-F(\Delta_i)\big\|\leq \frac{1}{k}.
\end{equation} 
\noindent
where, $x^{(i,k)}_j\in\Delta_j^{(i,k)}$ for any $i,k\in\mathbb{N}$ and $j=1,\dots,m_{i,k}$.
\noindent
By the spectral theorem, $\{E^{(i)}(\Delta_j^{i,k})\}_{j\leq m_{i,k}}\subset\mathcal{A}^W(F)$ for any $i,k\in\mathbb{N}$. Therefore, the von Neumann algebra $\mathcal{A}^W(D)$ generated by the set $D:=\{E^{(i)}(\Delta_j^{i,k}),\,j\leq m_{i,k},\,i,k\in\mathbb{N}\}$ is contained in $\mathcal{A}^W(F)$  
\begin{equation}\label{D1}
\mathcal{A}^W(D)\subset\mathcal{A}^W(F)=\mathcal{A}^W(O_2).
\end{equation}
\noindent
 Moreover, by (\ref{D}) 
\begin{equation*}
\mathcal{A}^C(O_2)\subset\mathcal{A}^C(D)\subset\mathcal{A}^W(F).
\end{equation*}
\noindent
where $\mathcal{A}^C(O_2)$ and $\mathcal{A}^C(D)$ are the $C^*$-algebras generated by $O_2$ and $D$ respectively.
\noindent
By the double commutant theorem, 
\begin{align*}
\mathcal{A}^W(F)&=[\mathcal{A}^C(O_2)]''\subset[\mathcal{A}^C(D)]''=\mathcal{A}^W(D)
\end{align*} 
\noindent
so that (see equation \ref{D1}),
\begin{equation}\label{double}
\mathcal{A}^W(D)=\mathcal{A}^W(F).
\end{equation} 

\noindent
By theorem 11, page 871 in Ref. \cite{Dunford}, there is a homeomorphism $\pi:\Lambda\to\pi(\Lambda)\subset\prod_{i=1}^{\infty}\{0,1\}$ which identifies the spectrum $\Lambda$ of $\mathcal{A}^C(D)$ with a closed subset of $\prod_{i=1}^{\infty}\{0,1\}$. Moreover, the function $f:\Lambda\to[0,1]$,

\begin{equation*}
f(\lambda):=\sum_{i=1}^{\infty}\frac{x_i}{3^i};\quad (x_1,\dots,x_n,\dots)=\pi(\lambda)
\end{equation*}
\noindent
is continuous and injective and then it distinguishes the points of $\Lambda$. Since $\Lambda$ and $[0,1]$ are Hausdorff, $f:\Lambda\to f(\Lambda)$ is a homeomorphism.


\noindent

\noindent
By the Gelfand-Naimark theorem and the spectral theorem for representations of commutative $C^*$-algebras, 
there is  an isometric $^*$-isomorphism between $\mathcal{A}^C(D)$ and $\mathcal{C}(\Lambda)$

\begin{align}\label{E}
T:\mathcal{C}(\Lambda)&\to \mathcal{A}^C(D)\subset B(\mathcal{H})\\
g&\mapsto T(g)=\int_{\Lambda}g(\lambda) d\widetilde{E}_{\lambda}.\notag
\end{align}
\noindent
where  $\widetilde{E}$ is the spectral measure from $\mathcal{B}(\Lambda)$ to $\mathcal{E(H)}$ corresponding to $T$.


\noindent
Since $f$ distinguishes the points of $\Lambda$, it generates $\mathcal{C}(\Lambda)$ and then 
\begin{equation*}
A=\int_{\Lambda} f(\lambda)\,d\widetilde{E}_{\lambda}
\end{equation*}
generates both $\mathcal{A}^C(D)$ and $\mathcal{A}^W(F)$.


\vskip.2cm
Now, we proceed to the proof of the existence of the weak Markov kernel $\widetilde{\nu}$ such that $(F,A,\widetilde{\nu})$ is a von Neumann triplet.

By (\ref{E}), for each $\Delta\in\mathcal{R(S)}$, there exists a continuous function $\gamma_{\Delta}\in\mathcal{C}(\Lambda)$ such that 
$$F(\Delta)=\int_{\Lambda}\gamma_{\Delta}(\lambda)\, d\widetilde{E}_{\lambda}.$$


\noindent
Let us consider the continuous function
\begin{equation*}
\nu_{\Delta}(t):=(\gamma_{\Delta}\circ f^{-1})(t),\quad\Delta\in\mathcal{R(S)}.
\end{equation*}
\noindent

\noindent

By the change of measure principle, we have,
\begin{align*}
F(\Delta)&=\int_{\Lambda}\gamma_{\Delta}(\lambda)\,d\widetilde{E}_{\lambda}=\int_{\Lambda}\gamma_{\Delta}(f^{-1}(f(\lambda)))\,d\widetilde{E}_{\lambda}\\
&=\int_{\sigma(A)}\gamma_{\Delta}(f^{-1}(t))\,d E_{t}=\int_{\sigma(A)}\nu_{\Delta}(t)\,d E_{t}=\nu_{\Delta}(A)
\end{align*} 
where $\sigma(A)=f(\Lambda)$ 
 and $E$ is the spectral measure corresponding to $A$ and defined by $E(\Delta)=\widetilde{E}(f^{-1}(\Delta))$, $\Delta\in\mathcal{B}(\sigma(A))$. 
\noindent
Therefore, for each $\Delta\in\mathcal{R(S)}$, $\nu_{\Delta}(f(\lambda))=\gamma_{\Delta}(\lambda)$, $\lambda\in\Lambda$, and $F(\Delta)=\nu_{\Delta}(A)$.  

\noindent
Moreover, for each $\lambda\in\sigma(A)$, the map $\nu_{(\cdot)}(\lambda):\mathcal{R(S)}\to[0,1]$ defines an additive set function. Indeed, let $\Delta\in\mathcal{R(S)}$ be the disjoint union of the sets $\Delta_1,\Delta_2\in\mathcal{R(S)}$. Then,
\begin{align*}
\int\nu_{(\Delta_1\cup\Delta_2)}(\lambda)\,dE_{\lambda}&=F(\Delta_1\cup\Delta_2)=F(\Delta_1)+F(\Delta_1)\\
&=\int\big[\nu_{\Delta_1}(\lambda)+\nu_{\Delta_2}(\lambda)\big]\,dE_{\lambda}
\end{align*}
so that, by the continuity of $\nu_{\Delta_1}(\lambda)$ and $\nu_{\Delta_2}(\lambda)$, we get 
\begin{equation*}
\nu_{\Delta_1}(\lambda)+\nu_{\Delta_2}(\lambda)=\nu_{(\Delta_1\cup\Delta_2)}(\lambda),\quad\forall\lambda\in\sigma(A).
\end{equation*}

Now, we extend $\nu$ to all $\mathcal{B}(X)$.

\noindent
Since  $A$ is the generator of $\mathcal{A}^W(F)$, for each $\Delta\in\mathcal{B}(X)$, there exists a Borel function $\omega_{\Delta}$ such that.
\begin{equation*}
F(\Delta)=\int_{\sigma(A)}\omega_{\Delta}(t)\,d E_t
\end{equation*}
\noindent
Then, we can consider the map $\widetilde{\nu}:\sigma(A)\times\mathcal{B}(X)\to[0,1]$ defined as follows
\begin{equation}\label{weakMK}
\widetilde{\nu}_{\Delta}(\lambda)=
\begin{cases}
\nu_{\Delta}(\lambda) & if \quad \Delta\in\mathcal{R(S)}\\
\omega_{\Delta}(\lambda) & if \quad \Delta\notin\mathcal{R(S)}.
\end{cases}
\end{equation}
\noindent
\noindent
Since $\widetilde{\nu}$ coincides with $\nu$ on $\mathcal{R(S)}$ it is additive on $\mathcal{R(S)}$. 

\noindent

\noindent
In order to prove that $\widetilde{\nu}$ is a weak Markov kernel, let us consider a set $\Delta\in\mathcal{B}(X)$ which is the disjoint union of the sets $\{\Delta_i\}_{i\in\mathbb{N}}$, $\Delta_i\in\mathcal{B}(X)$. Then,
\begin{align*}
\int\widetilde{\nu}_{(\cup_{i=1}^{\infty}\Delta_i)}(\lambda)\,dE_{\lambda}=\int\widetilde{\nu}_{\Delta}(\lambda)d E_\lambda=F(\Delta)=\sum_{i=1}^{\infty}F(\Delta_i)=\int\sum_{i=1}^{\infty}\widetilde{\nu}_{\Delta_i}(\lambda)\,dE_\lambda
\end{align*}
so that, 
\begin{equation*}
\sum_{i=1}^{\infty}\widetilde{\nu}_{\Delta_i}(\lambda)=\widetilde{\nu}_{\Delta}(\lambda), \quad E-a.e,
\end{equation*}
which implies that $\widetilde{\nu}:\sigma(A)\times\mathcal{B}(X)\to[0,1]$ is a weak Markov kernel. In particular $(F,A,\widetilde{\nu})$ is a von Neumann triplet.

Now, we proceed to prove the existence of the Markov kernel $\mu:\Gamma\times\mathcal{B}(X)\to[0,1]$ such that items 1, 2, and 3 of the theorem are satisfied. 


Since $X$ is Hausdorff locally compact second countable, it is a Polish space 
and, to each weak Markov kernel $\widetilde{\nu}:\sigma(A)\times\mathcal{B}(X)\to[0,1]$ such that $(F,A,\widetilde{\nu})$ is a von Neumann triplet, there corresponds a Markov kernel $\phi:\sigma(A)\times\mathcal{B}(X)\to[0,1]$ such that $(F,A,\phi)$ is a von Neumann triplet \cite{P,P0,B1}.
 Then, for each $\Delta\in\mathcal{B}(X)$,
\begin{equation*}
\int \widetilde{\nu}_{\Delta}(\lambda)\,dE_\lambda=F(\Delta)=\int \phi_{\Delta}(\lambda)\,dE_\lambda
\end{equation*}
\noindent
and,
\begin{equation}\label{ae}
\phi_{\Delta}(\lambda)=\widetilde{\nu}_{\Delta}(\lambda),\quad E-a.e.
\end{equation}
\noindent
By equation (\ref{ae}), for each $i\in\mathbb{N}$, there is a set $N_i\subset\sigma(A)$, $E(N_i)=\mathbf{0}$, such that 
\begin{equation}
\phi_{\Delta_i}(\lambda)=\widetilde{\nu}_{\Delta_i}(\lambda),\quad \lambda\in \sigma(A)-N_i.
\end{equation}
\noindent
Then, for each $i\in\mathbb{N}$,
\begin{equation}\label{17}
\phi_{\Delta_i}(\lambda)=\widetilde{\nu}_{\Delta_i}(\lambda),\quad\lambda\in\Gamma:=\sigma(A)-N
\end{equation}
\noindent
where, 
\begin{equation*}
N:=\cup_{i=1}^{\infty}N_i,\quad E(N)=\mathbf{0}.
\end{equation*}
\noindent

\noindent
Now, the map $\mu_{(\cdot)}(\cdot):\sigma(A)\times\mathcal{B}(X)\to[0,1]$,
\begin{equation*}
\mu_{(\cdot)}(\lambda)=
\begin{cases}
\widetilde{\nu}_{(\cdot)}(\lambda) & \lambda\in N\\
\phi_{(\cdot)}(\lambda) & \lambda\in\Gamma
\end{cases}
\end{equation*}
\noindent
is a strong Markov kernel since its restriction to $\Gamma$, $\phi_{(\cdot)}(\cdot):\Gamma\times\mathcal{B}(X)\to[0,1]$, is a Markov kernel.

\noindent
By (\ref{17}), 
$$\mu_{\Delta_i}(\lambda)=\widetilde{\nu}_{\Delta_i}(\lambda)$$
\noindent
so that, $\mu_{\Delta}$ is continuous for each $\Delta\in\mathcal{R(S)}$. We also have,
$$\mu_{\Delta}(A)=\phi_{\Delta}(A)=F(\Delta), \quad\Delta\in\mathcal{B}(X).$$
We have proved items 1, 2, and 3. Item 4 comes from theorem \ref{separate}. 

It remains to prove that $\mu$ is a Feller Markov kernel. By item 1, $\mu_{\Delta}$ is continuous for each $\Delta\in \mathcal{R(S)}$. Notice that for each open set $O\in\mathcal{B}(X)$, there is a countable family of sets $\Delta_i\in\mathcal{R(S)}$ such that $O=\cup_{i=1}^{\infty}\Delta_i$. Therefore, by theorem 2.2 in Ref. \cite{Billingsley}, and the continuity of $\mu_{\Delta}$ for each $\Delta\in\mathcal{R(S)}$, $\lim_{n\to\infty}\lambda_n=\lambda$ implies,
$$\lim_{n\to\infty}\int_X f(x)\,\mu_{dx}(\lambda_n)=\int_X f(x)\,\mu_{dx}(\lambda),\quad f\in\mathcal{C}_b(X)$$
\noindent
where, $\mathcal{C}_b(X)$ is the space of bounded, continuous real functions. Then, $G(\lambda):=\int f(x)\,\mu_{dx}(\lambda)$ is continuous whenever $f$ is continuous and $\mu$ is a Feller Markov kernel.

Finally, we note that $F(\Delta)=\mu_{\Delta}(A)$ implies the commutativity of $F$ and that ends the proof.

\end{proof}

In the proof of theorem \ref{weak} we have shown the existence of a Markov kernel $\phi$ such that $(F,A,\phi)$ is a von Neumann triplet. Then, we can state the following theorem.
\begin{theorem}\label{Markov}
A POVM $F:\mathcal{B}(X)\to\mathcal{F(H)}$ is commutative if and only if, there exists a bounded self-adjoint operator $A=\int \lambda\,dE_\lambda$ with spectrum $\sigma(A)\subset[0,1]$ and a Markov Kernel  $\phi:\sigma(A)\times\mathcal{B}(X)\to[0,1]$ such that 
$$F(\Delta)=\int_{\sigma(A)}\,\phi_{\Delta}(\lambda)\,dE_{\lambda},\quad\Delta\in\mathcal{B}(X).$$
\end{theorem}


\vskip1cm

\section{Characterization of POVMs which admit strong Feller Markov Kernels}
\noindent
\noindent
In the last section we proved that each commutative POVM admits a strong Markov kernel $\mu$ such that $\mu_{\Delta}$ is a continuous function for each $\Delta\in\mathcal{R(S)}$ where, $\mathcal{R(S)}$ is a ring which generates the Borel $\sigma$-algebra $\mathcal{B}(X)$.

\noindent
In the present section we characterize the commutative POVMs for which the Markov kernel $\mu$, whose existence was proved in theorem \ref{Markov}, is such that $\mu_{\Delta}$ is continuous for each $\Delta\in\mathcal{B}(X)$.
Whenever such a Markov kernel exists, we say that the POVM admits a strong Feller Markov kernel. In particular, we prove that a commutative POVM $F$ admits a strong Feller Markov kernel  if and only if $F$ is uniformly continuous.  

\begin{definition}
Let $F:\mathcal{B}(X)\to\mathcal{L}_s^+(\mathcal{H})$ be a POVM. 
$F$ is said to be uniformly continuous at $\Delta$ if, for any disjoint decomposition $\Delta=\cup_{i=1}^{\infty}\Delta_i$, 
$$\lim_{n\to\infty}\sum_{i=1}^{n}F(\Delta_i)=F(\Delta)$$
\noindent
 in the uniform operator topology. $F$ is said uniformly continuous if it is uniformly continuous at each $\Delta\in\mathcal{B}(X)$.
\end{definition}
\noindent
Notice that the term uniformly continuous derives from the fact that the $\sigma$-additivity of $F$  in the uniform operator topology is equivalent to the continuity in the uniform operator topology. Analogously, the $\sigma$-additivity of $F$ in the weak operator topology is equivalent to the continuity of $F$ in the weak operator topology \cite{Berberian}. 

\begin{proposition}[\cite{B12}]\label{down}
$F$ is uniformly continuous if and only if, 
$$\lim_{i\to\infty}\| F(\Delta_i)\|=0$$
\noindent
whenever $\Delta_i\downarrow\emptyset$.
\end{proposition}

\begin{definition}
A Markov kernel $\mu_{(\cdot)}(\cdot):\Lambda\times\mathcal{B}(X)\to[0,1]$ is said to be strong Feller if  $\mu_{\Delta}$ is a continuous function for each $\Delta\in\mathcal{B}(X)$.
\end{definition}
\begin{definition}
We say that a commutative POVM admits a strong Feller Markov kernel if there exists a strong Feller Markov kernel $\mu$ such that $F(\Delta)=\int \mu_{\Delta}(\lambda)\,dE_{\lambda}$, where $E$ is the sharp version of $F$.
\end{definition}

\noindent
In order to prove the main theorem of the section we need the following lemma.
\begin{lemma}\label{lem4}
Let $F$ be uniformly continuous. Let $\mu$ be a weak Markov kernel and $(F,A,\mu)$ a von Neumann triplet. Suppose that $\mu_{\Delta}$ is continuous for each $\Delta\in\mathcal{R(S)}$. Then, for each $\lambda\in\sigma(A)$, $\mu_{(\cdot)}(\lambda)$ is $\sigma$-additive on $\mathcal{R(S)}$.
\end{lemma}
\begin{proof}
Let $\Delta,\Delta_i\in\mathcal{R(S)}$, $\Delta_i\cap\Delta_j=\emptyset$, $\cup_{i=1}^{\infty}\Delta_i=\Delta$. Then, 
$\forall\epsilon>0$, there exists a number $\bar{n}\in\mathbb{N}$, such that for any $n>\bar{n}$,
\begin{align}\label{conti}
\|\mu_{\Delta}-\sum_{i=1}^n\mu_{\Delta_i}\|_{\infty}=\|\int\big(\mu_{\Delta}(\lambda)-\sum_{i=1}^n\mu_{\Delta_i}(\lambda)\big)\,dE_{\lambda}&\|\\
=\|F(\Delta)-F(\cup_{i=1}^n\Delta_i)&\|\leq\epsilon.\notag
\end{align}
Therefore, 
\begin{equation*}
|\mu_{\Delta}(\lambda)-\sum_{i=1}^n\mu_{\Delta_i}(\lambda)|\leq\epsilon,\,\,\,\,\forall\lambda\in\sigma(A).
\end{equation*}
\end{proof}
\vskip.3cm
\noindent

\noindent

\begin{theorem}\label{uni}
A commutative POVM $F:\mathcal{B}(X)\to\mathcal{L}_s^+(\mathcal{H})$ admits a strong Feller Markov kernel if and only if it is uniformly continuous.
\end{theorem}
\begin{proof}
\noindent
Suppose $F$ is uniformly continuous. By theorem \ref{weak}, there is a weak Markov kernel $\mu:\sigma(A)\times\mathcal{B}(X)\to[0,1]$ such that $\mu_{\Delta}(\cdot)$ is continuous for every $\Delta\in\mathcal{R(S)}$ and a self-adjoint operator $A$ such that $(F,A,\mu)$ is a von Neumann triplet.  By lemma \ref{lem4}, $\mu$ is $\sigma$-additive on $\mathcal{R(S)}$. By Charateodory's extension theorem the map $\mu:\sigma(A)\times\mathcal{R(S)}\to[0,1]$ can be extended  to a map $\widetilde{\mu}:\sigma(A)\times\mathcal{B}(X)\to[0,1]$ whose restriction to $\mathcal{R(S)}$ coincides with $\mu$ and such that $\widetilde{\mu}_{(\cdot)}(\lambda)$ is a probability measure for each $\lambda\in\sigma(A)$. Now we prove that $\widetilde{\mu}$ is a Markov kernel such that $F(\Delta)=\widetilde{\mu}_{\Delta}(A)$ and that $\widetilde{\mu}_{\Delta}$ is continuous for each $\Delta\in\mathcal{B}(X)$. We proceed by steps.

\textbf{1)} \textbf{Open sets}. Each open set $G$  is the union of a countable family of sets in $\mathcal{S}$, i.e., $G=\cup_{i=1}^{\infty}\Delta_i$, $\Delta_i\in \mathcal{S}$. Let us define the set $G_n:=\cup_{i=1}^{n}\Delta_i$. Therefore, $G_n\uparrow G$. Moreover, $\mu_{G_n}$ is continuous for each $n\in\mathbb{N}$, and
\begin{equation*}
u-\lim_{n\to\infty}F(G_n)=F(G).
\end{equation*}
Then, 
\begin{equation*}
F(G)=u-\lim_{i\to\infty}F(G_i)=u-\lim_{i\to\infty}\int\widetilde{\mu}_{G_i}(\lambda)\,dE_{\lambda}.
\end{equation*}
It follows that, $\forall\epsilon>0$, there exists a number $\bar{n}\in\mathbb{N}$, such that $n,m>\bar{n}$ implies,
\begin{align}\label{conti1}
\|\widetilde{\mu}_{G_n}-\widetilde{\mu}_{G_m}\|_{\infty}=\|\int[\widetilde{\mu}_{G_n}(\lambda)-\widetilde{\mu}_{G_m}(\lambda)]\,dE_{\lambda}&\|\\
=\|F(G_n)-F(G_m)&\|\leq\epsilon.\notag
\end{align}
Therefore,
\begin{equation}\label{Cauchy0}
|\widetilde{\mu}_{G_n}(\lambda)-\widetilde{\mu}_{G_m}(\lambda)|\leq\epsilon,\,\,\,\,\forall\lambda\in\sigma(A).
\end{equation}
Since $\widetilde{\mu}_{(\cdot)}(\lambda)$ is a probability measure, 
$$\lim_{i\to\infty}\widetilde{\mu}_{G_i}(\lambda)=\widetilde{\mu}_{G}(\lambda),\quad\forall\lambda\in\sigma(A).$$
\noindent
By equation (\ref{Cauchy0}), the convergence is uniform and this proves the continuity of $\widetilde{\mu}_{G}$. Moreover,
\begin{equation*}
F(G)=\lim_{i\to\infty}F(G_i)=\lim_{i\to\infty}\int\widetilde{\mu}_{G_i}(\lambda)\,dE_{\lambda}=\int\widetilde{\mu}_{G}(\lambda)\,dE_{\lambda}=\widetilde{\mu}_G(A).
\end{equation*}

\textbf{2)} \textbf{$\mathbf{G_{\delta}}$ sets}.  Let $G$ be a $G_{\delta}$-set. Then, there is 
a family of open sets $\{G_i\}_{i\in\mathbb{N}}$, $G\subset G_i$, such that $\cap_{i=1}^{\infty}G_i=G$. Then, 
by proceeding similarly to the step 1, one can prove the continuity of $\widetilde{\mu}_{G}$ and the equality $F(G)=\widetilde{\mu}_{G}(A)$.

\textbf{3)} \textbf{Borel sets}.  We use transfinite induction. 
Let $G_0$ be the family of open subsets of $X$, $\omega_1$ the first uncountable ordinal and $G_{\alpha}$, $\alpha<\omega_1$ the Borel hierarchy \cite{Kuratowski}. In particular, $G_1=G_{\delta},\, G_2=G_{\delta\sigma},\, G_3=G_{\delta\sigma\delta},\dots$ and  $G_{\alpha}=(\cup_{\beta<\alpha}G_{\beta})_{\sigma}$  for each limit ordinal $\alpha$. By means of the same reasoning that we used in items 1 and 2, one can prove the continuity of $\widetilde{\mu}_{\Delta}$ as well as that $\widetilde{\mu}_{\Delta}(A)=F(\Delta)$ whenever $\Delta$ is of the kind  $G_{\delta,\sigma},G_{\delta\sigma\delta}\dots$. Analogously, if $\widetilde{\mu}_{\Delta}$ is continuous for each $\Delta\in G_\alpha$ then, $\widetilde{\mu}_{\Delta}$ is continuous for each $\Delta$ in $G_{\alpha+1}$ and $\widetilde{\mu}_\Delta(A)=F(\Delta)$. Indeed, each set in $G_{\alpha+1}$ is either the countable union or the countable intersection of sets in $G_{\alpha}$ and the reasoning in items 1 and 2 can be used. If $\alpha$ is a limit ordinal and $\widetilde{\mu}_{\Delta}$ is continuous for each $\Delta\in G_\beta$, $\beta<\alpha$, then, $\widetilde{\mu}_{\Delta}$ is continuous for each $\Delta\in G_{\alpha}=(\cup_{\beta<\alpha}G_{\beta})_{\sigma}$ and $\widetilde{\mu}_\Delta(A)=F(\Delta)$. Indeed, each set in $G_{\alpha}$ is the countable union of sets in $\cup_{\beta<\alpha}G_{\beta}$ and the reasoning used in item 1 can be used. Therefore, by transfinite induction, $\widetilde{\mu}_{\Delta}$ is continuous for each $\Delta\in\cup_{\alpha<\omega_1}G_{\alpha}=\mathcal{B}(X)$ and $\widetilde{\mu}_{\Delta}(A)=F(\Delta)$.

\vskip.3cm
In order to prove the second part of the theorem we show that the existence of a strong Feller Markov kernel implies the uniform continuity of $F$. Suppose that there exists a strong Feller Markov kernel $\mu$ such that $F(\Delta)=\mu_{\Delta}(\lambda)$. Since $\mu$ is a Markov kernel it is $\sigma$-additive. Then,
\begin{equation*}
\lim_{n\to\infty}\big(\mu_{\Delta}(\lambda)-\sum_{i=1}^{n}\mu_{\Delta_i}(\lambda)\big)=0, \quad\lambda\in\sigma(A).
\end{equation*}
\noindent
where, $\Delta,\Delta_i\in\mathcal{B}(X)$, $\cup_{i=1}^{\infty}\Delta_i=\Delta$.

\noindent
By hypothesis,
$$\mu_{\Delta}(\lambda)-\sum_{i=1}^{n}\mu_{\Delta_i}(\lambda)\in\mathcal{C}(\sigma(A)),\quad\forall n\in\mathbb{N}.$$
\noindent
Then, by theorem \ref{lem7} in appendix B,
$$u-\lim_{n\to\infty}\big(\mu_{\Delta}(\lambda)-\sum_{i=1}^{n}\mu_{\Delta_i}(\lambda)\big)=0.$$
\noindent
so that 
\begin{align*}
\lim_{n\to\infty}\| F(\Delta)-F(\cup_{i=1}^n\Delta_i)\|=\lim_{n\to\infty}\|\mu_{\Delta}-\sum_{i=1}^n\mu_{\Delta_i}\|_{\infty}=0.
\end{align*} 
\noindent
\end{proof}

\begin{example}\label{example2}
Let us consider the following unsharp position observable 
\begin{align}\label{PositionExample2}
Q^f(\Delta)&:=\int_{[0,1]}\mu_{\Delta}(\lambda)\,dQ_\lambda,\quad\Delta\in\mathcal{B}(\mathbb{R}),\\
\mu_{\Delta}(\lambda)&:=\int_{\mathbb{R}}\chi_{\Delta}(\lambda-y)\, f(y)\,dy,\quad \lambda\in[0,1]\notag
\end{align}
\noindent
where, $f$ is a bounded, continuous function such that $f(y)=0$, $y\notin [0,1]$ and  

$$\int_{[0,1]} f(y)\,dy=1,$$
\noindent
 and $Q_\lambda$ is the spectral measure corresponding to the position operator
\begin{align*}
Q:L^2([0,1])&\to L^2([0,1])\\
(Q\psi)(x)&:=x\psi(x)
\end{align*}
\noindent
for almost all $x\in [0,1]$. Notice that, for each $\Delta\in\mathcal{B}({\mathbb{R}})$, $\mu_{\Delta}:[0,1]\to[0,1]$ is continuous. Indeed, by the uniform continuity of $f$, for each $\epsilon>0$, there is a $\delta>0$ such that $\vert \lambda-\lambda'\vert\leq\delta$ implies $\vert f(\lambda-y)-f(\lambda'-y)\vert\leq\epsilon$, for each $y$. Therefore,
\begin{align*}
\vert\mu_{\Delta}(\lambda)-\mu_{\Delta}(\lambda')\vert&=\Big\vert\int_{\mathbb{R}}\chi_{\Delta}(\lambda-y)\, f(y)\,dy-\int_{\mathbb{R}}\chi_{\Delta}(\lambda'-y)\, f(y)\,dy\Big\vert\\
&=\Big\vert\int_{\Delta} [f(\lambda-y)-f(\lambda'-y)]\,dy\Big\vert\leq\epsilon\int_{\Delta\cap[-1,1]}\,dy\leq2\epsilon
\end{align*}
\noindent
By theorem \ref{uni} and the continuity of $\mu_{\Delta}$, $\Delta\in\mathcal{B}(\mathbb{R})$, $Q^f$ is uniformly continuous. That can be proved as follows. Suppose $\Delta_i\downarrow\Delta$ and $f(y)\leq M$, $y\in\mathbb{R}$. Since, for each $\lambda\in[0,1]$, 
$$\mu_{\Delta_i-\Delta}(\lambda)=\int_{\Delta_i-\Delta} f(\lambda-y)\,dy\leq M\int_{(\Delta_i-\Delta)\cap[-1,1]}dy$$
\noindent
we have that, for each $\psi\in L^2([0,1])$, $\|\psi\|^2=1$,
\begin{align*}
\langle\psi,Q^f(\Delta_i-\Delta)\psi\rangle=\int_{[0,1]}\mu_{\Delta_i-\Delta}(x)\,|\psi(x)|^2\,dx\leq M\int_{(\Delta_i-\Delta)\cap[-1,1]}dy
\end{align*}
\noindent
which proves the uniform continuity of $Q^f$.
\end{example}

\noindent
In the case of  uniformly continuous POVMs, we can prove a necessary condition for the norm-1-property.
\begin{definition}[\cite{Heinonen}]
A POVM $F$ has the norm-1-property if $\|F(\Delta)\|=1$, for each $\Delta\in\mathcal{B}(X)$ such that $F(\Delta)\neq\mathbf{0}$. 
\end{definition}
\begin{theorem}\label{norm1}
Let $F$ be uniformly continuous. Then, $F$ has the norm-1-property only if $\|F(\{\lambda\})\|\neq 0$ for each $\lambda\in\sigma(F)$.
\end{theorem}
\begin{proof}
We proceed by contradiction. Suppose that $F$ has the norm-1 property and that there exists $\lambda\in\sigma(F)$, such that $F(\{\lambda\})=\mathbf{0}$. Let $\Delta_i$ be a decreasing family of open sets such that, $\cap_{i=1}^\infty\Delta_i=\{\lambda\}$. The existence of such family is assured by the local compactness of $X$. 
Since $\lambda\in\sigma(F)$ and $\lambda\in\Delta_i$, we have $F(\Delta_i)\neq\mathbf{0}$ for any $i\in\mathbb{N}$ (see Definition \ref{spectrum}) and, by the norm-1 property, $\|F(\Delta_i)\|=1$. 
By the uniform continuity of $F$ and proposition \ref{down},
\begin{align*}
1=\lim_{i\to\infty}\|F(\Delta_i)\|&=\lim_{i\to\infty}\|F(\Delta_i)-F(\{\lambda\})+F(\{\lambda\})\|\\
&\leq\lim_{i\to\infty}\|F(\Delta_i-\{\lambda\})\|+\|F(\{\lambda\})\|=0.
\end{align*}

\end{proof}
\begin{example}
Let $Q^f$ be as in example \ref{example2}. Theorem \ref{norm1} implies that $Q^f$ cannot have the norm-1 property. Indeed, for each $\lambda\in\mathbb{R}$, and $\lambda_i\mapsto \lambda$, 
$$\big(Q^f(\{\lambda\})\psi\big)(x)=\lim_{i\to\infty}\big(Q^f([\lambda,\lambda_i))\psi\big)(x)=\lim_{i\to\infty}\mu_{[\lambda,\lambda_i)}(x)\psi(x)=0,$$ 
\noindent
for all $\psi\in L^2([0,1])$ and almost all $x\,\in [0,1]$.
\end{example}

\noindent
We refer to \cite{B12} for an analysis of the relevance of theorem \ref{norm1} to the problem of localization of massless relativistic particles.
\vskip1cm

\section{POVMs that are norm bounded by scalar measures}

\noindent
In the present section, we prove that a commutative POVM which is norm bounded by a scalar measure admits a strong Feller Markov kernel. Then, we apply the result to the case of the unsharp position observable. 
\begin{definition}\cite{Schroeck,Schroeck1}
A POVM $F:\mathcal{B}(X)\to\mathcal{F(H)}$ is norm bounded by a measure $\nu:\mathcal{B}(X)\to[0,1]$ if there exists a positive number $c$ such that $\| F(\Delta)\|\leq c\,\nu(\Delta)$, for each $\Delta\in\mathcal{B}(X)$.
\end{definition}
\begin{theorem}\label{abs}
Let $F$ be norm bounded by a finite measure $\nu$. Then, $F$ is uniformly continuous.
\end{theorem}
\begin{proof}
Suppose $\Delta_i\downarrow\emptyset$. We have 
\begin{align*}
\lim_{i\to\infty}\|F(\Delta_i)\|\leq c\lim_{i\to\infty}\nu(\Delta_i)=0.
\end{align*} 
\noindent
Proposition \ref{down} ends the proof.
\end{proof}
\begin{corollary}
Let $F$ be norm bounded by a finite  measure $\nu$. Then, $F$ is commutative if and only if there exist a self-adjoint operator $A$ and a strong Feller Markov kernel $\mu:\sigma(A)\times\mathcal{B}(X)\to[0,1]$ such that:
\begin{equation}
F(\Delta)=\mu_{\Delta}(A),\quad\Delta\in\mathcal{B}(X)
\end{equation} 
\end{corollary}
\begin{proof}
By theorem \ref{abs}, $F$ is uniformly continuous. Then, theorem \ref{uni} implies the thesis. 
\end{proof}

\begin{example}\label{PositionExample3}
Let $Q^f$ be the unsharp position POVM defined in example \ref{example2}


Then, $Q^f$ is norm bounded by the measure 
$$\nu(\Delta)=M\int_{\Delta\cap[-1,1]}dy.$$ 
\noindent
Indeed, for each $\psi\in L^2([0,1])$, $\|\psi\|^2=1$,
\begin{align*}
\langle\psi,Q^f(\Delta)\psi\rangle=\int_{[0,1]}\mu_{\Delta}(x)\,\vert\psi(x)\vert^2\,dx\leq M\int_{\Delta\cap[-1,1]}dy
\end{align*}
\noindent
where, the inequality
$$\mu_{\Delta}(x)=\int_{\Delta} f(x-y)\,dy\leq M\int_{\Delta\cap[-1,1]} dy$$
\noindent
has been used.

\noindent
Therefore, by theorem \ref{abs}, $Q^f(\Delta)$ is uniformly continuous. 
\end{example}
\subsection{Unsharp Position Observable}\label{PositionMomentum}
\noindent
In the present subsection, we study an important kind of norm bounded POVMs, the unsharp position observables obtained as the marginals of a covariant phase space observable. 

\noindent
In the following $\mathcal{H}=L^2(\mathbb{R})$, $Q$ and $P$ denote position and momentum observables respectively and $\ast$ denotes convolution, i.e. $(f\ast g)(x)=\int f(y)g(x-y)d y$. 

\noindent
Let us consider the joint position-momentum POVM \cite{Ali,Bush,Davies,Guz,Holevo1,Prugovecki,Schroeck1,Stulpe} 
\begin{equation*} \label{phase}
F(\Delta\times\Delta')=\int_{\Delta\times\Delta'}U_{q,p}\,\gamma\,U^*_{q,p}\,d q\, d p
\end{equation*}
where, $U_{q,p}=e^{-iqP}e^{ipQ}$ and $\gamma=\vert f\rangle\langle f\vert$, $f\in L^2(\mathbb{R})$, $\|f\|_2=1$.
The marginal
\begin{equation}
\label{approximate}
Q^{f}(\Delta):=F(\Delta\times\mathbb{R})=\int_{-\infty}^{\infty}(\chi_{\Delta}\ast\vert f\vert^2)(\lambda)\,dQ_\lambda,\quad\Delta\in\mathcal{B}(\mathbb{R}),
\end{equation}
is an unsharp position observable.  Notice that the map $\mu_{\Delta}(\lambda):=(\chi_{\Delta}\ast\vert f\vert^2)(\lambda)$ defines a Markov kernel. 

\noindent
Moreover, $Q^f$ is norm bounded by the Lebesgue measure. Indeed,
\begin{align*} 
Q^f(\Delta)=F(\Delta\times\mathbb{R})&=\int_{\Delta\times\mathbb{R}}U_{q,p}\,\gamma\,U^*_{q,p}\,d q\, d p\\
&=\int_{\Delta}\,dq\int_{\mathbb{R}} U_{q,p}\,\gamma\,U^*_{q,p}\,d p\\
&=\int_{\Delta}\widehat{Q}(q)\,dq\leq\int_{\Delta}\mathbf{1}\,dq
\end{align*}

\noindent
where, 
$$\widehat{Q}(q)=\int_{\mathbb{R}}U_{q,p}\,\gamma\,U^*_{q,p}\,dp.$$
\noindent
Although $Q^f$ is norm bounded by the Lebesgue measure on $\mathbb{R}$, it is not uniformly continuous. That does not contradict theorem \ref{abs} since the Lebesgue measure on $\mathbb{R}$ is not finite. Anyway, $Q^f$ is uniformly continuous on each Borel set $\Delta$ with finite Lebesgue measure. 

\noindent
Now, we show that $Q^f$ is not in general uniformly continuous. We give the details of the following particular case.

\begin{example}[Optimal Phase Space Representation]\label{optimal}
\noindent
If we choose 
$$f^2(y)=\frac{1}{l\,\sqrt{2\,\pi}}\,e^{(-\frac{y^2}{2\,l^2})},\quad l\in\mathbb{R}-\{0\}.$$ 
\noindent
in (\ref{approximate}), we get an optimal phase space representation of quantum mechanics \cite{Prugovecki}. In this case,

\begin{align*}
Q^{f}(\Delta)&=\int_{-\infty}^{\infty}\Big(\int_{\Delta}\vert f(\lambda-y)\vert^2)\,d y\Big)\,dQ_\lambda\\
&=\frac{1}{l\,\sqrt{2\,\pi}}\int_{-\infty}^{\infty}\Big(\int_{\Delta}e^{-\frac{(\lambda-y)^2}{2\,l^2}}\,d y\Big)\,dQ_\lambda=\int_{-\infty}^{\infty}\mu_{\Delta}(\lambda)\,dQ_\lambda
\end{align*}
\noindent
where, 
\begin{equation}\label{PM}
\mu_{\Delta}(\lambda)=\frac{1}{l\,\sqrt{2\,\pi}}\int_{\Delta}e^{-\frac{(\lambda-y)^2}{2\,l^2}}\,d y 
\end{equation}
\noindent
defines a Markov kernel. 

In order to prove that $Q^f$ is not uniformly continuous we consider the family of sets $\Delta_i=(-\infty, a_i)$, $\lim_{i\to \infty}a_i=-\infty$ such that $\Delta_i\downarrow\emptyset$, and prove that $\lim_{i\to\infty}\|Q^f(\Delta_i)\|=1$. For each  $i\in\mathbb{N}$,
\begin{align*}
\lim_{\lambda\to-\infty}\mu_{\Delta_i}(\lambda)&=\lim_{\lambda\to-\infty}\frac{1}{l\,\sqrt{2\,\pi}}\int_{\Delta_i}e^{-\frac{(\lambda-y)^2}{2\,l^2}}\,d y\\
&=\lim_{\lambda\to-\infty}\frac{1}{l\,\sqrt{2\,\pi}}\int_{(-\infty,\, a_i-\lambda)}e^{-\frac{y^2}{2\,l^2}}\,d y=\frac{1}{l\,\sqrt{2\,\pi}}\int_{-\infty}^{\infty}e^{-\frac{y^2}{2\,l^2}}\,d y=1.  
\end{align*}
\noindent
Now, we prove that $\|Q^f(\Delta_i)\|=1$, $i\in\mathbb{N}$. Indeed, if 
$$\psi_n(x)=\chi_{[-n,-n+1]}(x),$$
\begin{align}\label{n}
\lim_{n\to\infty}\langle\psi_n,Q^f(\Delta_i)\psi_n\rangle&=\lim_{n\to\infty}\int_{-\infty}^{\infty}\mu_{\Delta_i}(x)\vert\psi_n(x)\vert^2\,dx\\
&=\lim_{n\to\infty}\int_{[-n,-n+1]}\mu_{\Delta_i}(x)\,dx=1.
\end{align} 
\noindent
Since, for each $\Delta\in\mathcal{B}(\mathbb{R})$, $\|Q^f(\Delta)\|\leq 1$, equation (\ref{n}) implies that $\|Q^f(\Delta_i)\|=1$, for each $i\in\mathbb{N}$. Hence,  $\lim_{i\to\infty}\|Q^f(\Delta_i)\|=1$ and $Q^f$ cannot be uniformly continuous.

It is worth noticing that although $Q^f$ is not uniformly continuous, $\mu_{\Delta}$ is continuous for each interval $\Delta\in\mathcal{B}(\mathbb{R})$. Indeed, 
\begin{align*}
\vert\mu_{\Delta}(\lambda)-\mu_{\Delta}(\lambda')\vert&=\frac{1}{l\,\sqrt{2\,\pi}}\Big\vert \int_{\Delta}e^{-\frac{(\lambda-y)^2}{2\,l^2}}\,dy-\int_{\Delta}e^{-\frac{(\lambda'-y)^2}{2\,l^2}}\,dy \Big\vert\\
&=\frac{1}{l\,\sqrt{2\,\pi}}\Big\vert\int_{\Delta_\lambda}e^{-\frac{(y)^2}{2\,l^2}}\,dy-\int_{\Delta_{\lambda'}} e^{-\frac{(y)^2}{2\,l^2}}\,dy\Big\vert\leq\frac{1}{l\,\sqrt{2\,\pi}}\Big\vert\int_{\overline{\Delta}} e^{-\frac{(y)^2}{2\,l^2}}\,dy\Big\vert
\end{align*}
\noindent
where, 
$$\Delta_\lambda=\{z\in\mathbb{R}\,\vert\,z=y-\lambda,\,y\in\Delta\},\quad\Delta_{\lambda'}=\{z\in\mathbb{R}\,\vert\,z=y-\lambda',\,y\in\Delta\}$$ 
\noindent
and,
$$\overline{\Delta}=(\Delta_\lambda-\Delta_{\lambda'})\cup(\Delta_{\lambda'}-\Delta_\lambda).$$

\noindent
Therefore, $\vert \lambda-\lambda'\vert\leq\epsilon$ implies,
\begin{equation*}
\vert\mu_{\Delta}(\lambda)-\mu_{\Delta}(\lambda')\vert\leq\frac{1}{l\,\sqrt{2\,\pi}}\Big\vert\int_{\overline{\Delta}} e^{-\frac{(y)^2}{2\,l^2}}\,dy\Big\vert\leq\frac{1}{l\,\sqrt{2\,\pi}}\,\int_{\overline{\Delta}}\,dy=\frac{\sqrt{2}}{l\,\sqrt{\pi}}\,\epsilon.
\end{equation*}
\end{example}

\section{Conclusions}
\noindent
We already pointed out that although the set of commutative POVMs is a particular subset of the set of POVMs, the commutative POVMs are relevant both from the mathematical and the physical viewpoint. It is well known that they can be interpreted as the smearing of real PVMs, $E$, and that the smearing can be realized by means of  Markov kernels, $\mu$. Anyway no characterization of the smearing (the Markov kernel) is known. In the present paper such a characterization is given and its mathematical and physical implications are analyzed. For example we answered the following questions: 1) Can the smearing be realized by means of a Feller Markov kernel?, 2) What can we say about the continuity of the functions $\mu_{\Delta}$?, 3) Can the smearing be realized by means of a strong Feller Markov kernel?, 4) What is the physical interpretation of the smearing when it is realized by means of a strong Feller Markov kernel?, 5) Is the smearing able to distinguish the points in the spectrum of  the PVM $E$?, 6) Are there physical examples that can be used as an illustration of items 1) to 5) above?   
In order to answer such questions, we had to provide a new and stronger characterization of a commutative POVM $F$ as the smearing of a real PVM $E$.


\appendix

\section{$\mathcal{A}^W(F)$ coincides with the von Neumann algebra generated by $\{F(\Delta)\}_{\Delta\in\mathcal{R(S)}}$}
\noindent
We recall that $\mathcal{S}\subset\mathcal{B}(X)$ is a countable basis for the topology of $X$ and $\mathcal{R(S)}$ is the ring generated by $\mathcal{S}$. 
\begin{proof}
\noindent
Let $M:=\{F(\Delta)\}_{\Delta\in\mathcal{B}(X)}$, and $\mathcal{A}^W(F)=\mathcal{A}^W(M)$ the von Neumann algebra generated by $F$. Let $G$ denote the family of open subsets of $X$ and $O:=\{F(\Delta),\,\,\Delta\in G\}$. We have $\mathcal{A}^W(O)\subset\mathcal{A}^W(F)$. Since the POVM $F$ is regular, $O$ is dense in $M$ and $\mathcal{A}^W(F)\subset\mathcal{A}^W(O)$. Therefore, 
\begin{equation}\label{O}
\mathcal{A}^W(F)=\mathcal{A}^W(M)=\mathcal{A}^W(O).
\end{equation}
\noindent
 Now, we prove that the von Neumann algebra $\mathcal{A}^W(O_1)$ generated by the set $O_1=\{F(\Delta)\}_{\Delta\in\mathcal{S}}$ coincides with $\mathcal{A}^W(O)$.

\noindent
For each open set $G$, there exists a family of sets $\{\Delta_i\}_{i\in\mathbb{N}}\subset \mathcal{S}$, such that $G=\cup_{i=1}^{\infty}\Delta_i$. Let $G_n=\cup_{i=1}^n\Delta_i$. Then, $G_n\uparrow G$ and 
\begin{align*}
F(G)=\lim_{n\to\infty}F(G_n)=\lim_{n\to\infty}F(\cup_{i=1}^n\Delta_i). 
\end{align*}
Since the von Neumann algebra generated by $O_1$ contains $F(\cup_{i=1}^n\Delta_i)$ for each $n\in\mathbb{N}$, it must contain $F(G)=\lim_{n\to\infty}F(\cup_{i=1}^n\Delta_i)$. Therefore, $\mathcal{A}^W(O)\subset\mathcal{A}^W(O_1)$. Moreover, $\mathcal{A}^W(O_1)\subset\mathcal{A}^W(O)$ since $O_1\subset O$. Then, $\mathcal{A}^W(O)=\mathcal{A}^W(O_1)$ and  by equations (\ref{O}),
\begin{equation}\label{W}
\mathcal{A}^W(O_1)=\mathcal{A}^W(O)=\mathcal{A}^W(F).
\end{equation}
\noindent
Since $\{F(\Delta)\}_{\Delta\in\mathcal{S}}\subset\{F(\Delta)\}_{\Delta\in\mathcal{R(S)}}\subset\{F(\Delta)\}_{\Delta\in\mathcal{B}(X)}$, the von Neumann algebra generated by the set $\{F(\Delta)\}_{\Delta\in\mathcal{R(S)}}$ must coincide with  $\mathcal{A}^W(F)$.

\end{proof}

\appendix

\section*{Appendix B: Sequences of continuous functions}
\renewcommand{\thetheorem}{B\arabic{theorem}}
\setcounter{theorem}{0}
\setcounter{equation}{0}
\noindent
The following theorem is due to Dini. We give a proof based on the use of sequences.
\begin{theorem}\label{lem7} 
Let $\{f_n(\lambda)\}_{n\in\mathbb{N}}$ be a non increasing sequence of continuous functions defined on a compact set $B\subset[0,1]$ with values in $[0,1]$ and such that $f_n(\lambda)\to 0$ point-wise. Then, $f_n(\lambda)\to 0$ uniformly.
\end{theorem}
\begin{proof}
Since $f_{n+1}(\lambda)\leq f_n(\lambda)$ for each $\lambda\in B$, we have $\|f_{n+1}\|_{\infty}\leq\|f_{n}\|_{\infty}$. If $\|f_{n}\|_{\infty}\to 0$ clearly $f_n(\lambda)\to 0$ uniformly.

\noindent 
Then, suppose $\|f_{n}\|_{\infty}\to a>0$. Since $\|f_{n+1}\|_{\infty}\leq\|f_{n}\|_{\infty}$, we have $\|f_{n}\|_{\infty}\geq a$, for each $n\in\mathbb{N}$.

\noindent
 Let $\lambda_n$ be such that $f_n(\lambda_n)=\|f_n\|_{\infty}$. Since $\{\lambda_n\}$ is a bounded sequence of real numbers, there exists a convergent subsequence $\{\lambda_{n_k}\}_{k\in\mathbb{N}}$. Let $\beta$ be its limit, i.e., $\beta:=\lim_{k\to\infty}\lambda_{n_k}$. The compactness of $B$ assures that $\beta\in B$. Moreover, $\lim_{k\to\infty}f_{n_k}(\lambda_{n_k})=a$. 

\noindent
Let us consider the sequence of numbers $f_{n_k}(\beta)$. We prove that $f_{n_k}(\beta)\geq a$ for each $k\in\mathbb{N}$. We proceed by contradiction. Suppose that there exists $\bar{k}\in\mathbb{N}$ such that $f_{n_{\bar{k}}}(\beta)<a$. Then, there exists a neighborhood $I(\beta)$ of $\beta$ such that $f_{n_{\bar{k}}}(\lambda)<a$ for each $\lambda\in I(\beta)$. Moreover, since $\lambda_{n_k}\to \beta$, there exists $l\in\mathbb{N}$ such that $k>l$ implies $\lambda_{n_k}\in I(\beta)$. Take $k>\max\{\bar{k},l\}$. Then, $\lambda_{n_k}\in I(\beta)$ and $f_{n_k}(\lambda)\leq f_{n_{\bar{k}}}(\lambda)$, for each $\lambda\in B$. Therefore,  
$$f_{n_k}(\lambda_{n_k})\leq f_{n_{\bar{k}}}(\lambda_{n_k})<a$$
which contradicts the fact that $f_{n_k}(\lambda_{n_k})=\|f_{n_k}\|_{\infty}\geq a$, for each $k\in\mathbb{N}$.

\noindent
We have proved that $f_{n_k}(\beta)\geq a$, for each $k\in\mathbb{N}$. This implies that $\lim_{k\to\infty}f_{n_k}(\beta)\geq a$ and contradicts one of the hypothesis of the lemma, i.e., $\lim_{n\to\infty}f_n(\lambda)=0$ for each $\lambda\in B$.
\end{proof}

\noindent
\textbf{Aknowledgements}:  The author is grateful to Prof. Paul Busch for his encouragement by confirming, in a private communication, the relevance of commutative POVMs in quantum mechanics and quantum information theory also in the light of his recent work on measurement uncertainty relations. The author is grateful to an anonymous referee for his helpful suggestions that improved the paper.   
The present work was performed under the auspices of G.N.F.M (Gruppo Nazionale di Fisica Matematica).


\end{document}